\documentclass[12pt,reqno,a4paper]{amsart}

\usepackage{amsfonts}
\usepackage{epsfig}
\usepackage{graphicx}
\usepackage{amsmath}
\usepackage{amssymb}
\DeclareMathAlphabet{\mathpzc}{OT1}{pzc}{m}{it}

\newcounter{main}

\numberwithin{equation}{section}

\newtheorem{theorem}{Theorem}[section]

\newtheorem{lemma}[theorem]{Lemma}

\newtheorem{maintheorem}{Theorem}

\newcommand{\blanksquare}{\,\,\,$\sqcup\!\!\!\!\sqcap$}

\newcounter{example}
{{\stepcounter{example}}{\flushleft {\bf Example \arabic{example}:}}}%
{\par}

\title[Topological stability for conservative systems]
{Topological stability for conservative systems}

\author[M. Bessa]{M\'{a}rio Bessa}
\address{Departamento de Matem\'atica da Universidade do Porto, 
Rua do Campo Alegre, 687, 
4169-007 Porto, Portugal \\ ESTGOH-Instituto Polit\'ecnico de Coimbra, Rua General Santos Costa, 3400-124 Oliveira do Hospital, Portugal}
\email{bessa@fc.up.pt}
\author[J. Rocha]{Jorge Rocha}
\address{Departamento de Matem\'atica da Universidade do Porto, 
Rua do Campo Alegre, 687, 
4169-007 Porto, Portugal}
\email{jrocha@fc.up.pt}

\begin{document}

\begin{abstract}
We prove that the $C^1$ interior of the set of all topologically stable $C^1$ incompressible flows is contained in the set of Anosov incompressible flows. Moreover, we obtain an analogous result for the discrete-time case.
\end{abstract}

\maketitle

\noindent\emph{MSC 2000:} primary 37C10;37C15; secondary 37C05, 37C27.\\
\emph{keywords:} Volume-preserving flows and diffeomorphisms, topological stability, Anosov systems.\\

\begin{section}{Introduction: basic definitions and statement of the results}

We consider an $n$-di\-men\-sio\-nal ($n\geq 3$) closed and connected $C^\infty$ Riemaniann manifold $M$ endowed with a volume-form $\omega$. Let $\mu$ denote the measure associated to $\omega$, that we call Lebesgue measure, and let $d(\cdot,\cdot)$ denote the metric induced by the Riemannian structure. We say that a vector field $X\colon M\rightarrow TM$ is \emph{divergence-free} if $\nabla\cdot X=0$ or equivalently if the measure $\mu$ is invariant for the associated flow, $X^t\colon M\rightarrow M$, $t \in \mathbb{R}$. In this case we say that the flow is \emph{incompressible} or \emph{volume-preserving}. We denote by $\mathfrak{X}_\mu^r(M)$ ($r\geq 1$) the space of $C^r$ divergence-free vector fields on $M$ topologized with the usual $C^r$ Whitney topology.

Given $X\in\mathfrak{X}_\mu^{1}(M)$ let $Sing(X)$ denote the set of \emph{singularities} of $X$ and $\mathpzc{R}=\mathpzc{R}(M):=M\setminus Sing(X)$ the set of \emph{regular} points. We say that $\sigma\in Sing(X)$ is a \emph{hyperbolic singularity} if $DX_\sigma$ has no eigenvalue with null real part.

Let $X,Y\in\mathfrak{X}_\mu^1(M)$; $Y$ is \emph{semiconjugated} to $X$ if the flow associated to $Y$ is semiconjugated to the one of $X$, i.e., there exists a continuous and onto map $h\colon M\rightarrow M$ and a continuous real map $\tau \colon M\times\mathbb{R}\rightarrow\mathbb{R}$ such that
\begin{enumerate}
 \item [a)] for any $x\in M$, $\tau_x \colon \mathbb{R}\rightarrow\mathbb{R}$ is an orientation preserving homeomorphism where $\tau(x,0)=0$ and
 \item [b)] for all $x\in M$ and $t\in\mathbb{R}$ we have $h(Y^t(x))=X^{\tau(x,t)}(h(x))$.
\end{enumerate}

We say that $X\in\mathfrak{X}_\mu^{1}(M)$ is \emph{topologically stable} in $\mathfrak{X}_\mu^{1}(M)$ if for any $\epsilon>0$, there exists $\delta>0$ such that for any $Y\in \mathfrak{X}_\mu^{1}(M)$ $\delta$-$C^0$-close to $X$, there exists a semiconjugacy from $Y$ to $X$, i.e., there exists $h\colon M\rightarrow M$ and $\tau \colon M\times\mathbb{R}\rightarrow\mathbb{R}$ satisfying a) and b) above, and $d(h(x),x)<\epsilon$ for all $x\in M$. It is worth to emphasize that our definition of topological stability is restricted to the conservative setting and not to the  broader space of dissipative flows. We denote the set of topological stable incompressible flows by $\mathpzc{TS}_\mu(M)$.

A vector field is said to be \emph{Anosov} if the tangent bundle $TM$
splits into three continuous $DX^t$-invariant nontrivial subbundles
$E^0\oplus E^1\oplus E^2$ where $E^0$ is the flow direction, the
sub-bundle $E^2$ is uniformly contracted and the
sub-bundle $E^1$ is uniformly expanded by $DX^t$
for $t>0$. Of course that, for an Anosov flow, we have $Sing(X)=\emptyset$ which follows from the fact that the dimensions of the subbundles are constant on the entire manifold.

For analogous definitions in the volume-preserving diffeomorphisms context see \S\ref{diff}.

The concept of topological stability was first introduced by Walters. In (\cite{W}) he proved that Anosov diffeomorphisms are topologically stable. In (\cite{Ni}) Nitecki proved that topological stability was a necessary condition to get Axiom A plus strong transversality. Later, in (\cite{Ro2}), Robinson proved that Morse-Smale flows are topologically stable. In the mid 1980's, Hurley obtained necessary conditions for topological stability (see \cite{H,H2,H3}). About ten years ago it was proved by Moriyasu, Sakai and Sumi  (see~\cite{MSS}) that, if $X$ is a vector field in the $C^1$ interior of the set of topologically stable vector fields (in the broader space of dissipative flows) then $X$ satisfies  the Axiom A and the strong transversality properties. Our main result (Theorem~\ref{teo1}) is a generalization of the main theorem in (\cite{MSS}) for divergence-free vector fields. Although this result is expectable, its proof uses perturbations techniques that only recently become available.

Given a set $A\subset \mathfrak{X}^1_\mu(M)$ let $int_{C^1}(A)$ denote the interior of $A$ in $\mathfrak{X}^1_\mu(M)$ with respect to the $C^1$-topology.

\begin{maintheorem}\label{teo1}
If $X\in int_{C^1}(\mathpzc{TS}_\mu(M))$ then $X$ is Anosov.
\end{maintheorem}

Due to well-know results (see~\cite{PT}) about the restrictions of the existence of manifolds supporting Anosov flows, we obtain that, for general manifolds, the $C^1$-interior of topological stable incompressible flows must be empty. It is also interesting to note that, in the geodesic flow context, Anosov systems are not necessarily topological stable (see~\cite{R}).

\bigskip

Given $x\in \mathpzc{R}(X)$ we consider its normal bundle $N_{x}=X(x)^{\perp}\subset T_{x}M$ and define the \emph{linear Poincar\'{e} flow} by $P_{X}^{t}(x):=\Pi_{X^{t}(x)}\circ DX^{t}_{x}$ where $\Pi_{X^{t}(x)}:T_{X^{t}(x)}M\rightarrow N_{X^{t}(x)}$ is the projection along the direction of $X(X^{t}(x))$. Let $\Lambda \subset R$ be an $X^t$-invariant and compact set. We say that $\Lambda$ is a \emph{hyperbolic} set for the linear Poincar\'{e} flow if there exist constants $\lambda\in (0,1)$, $C>0$ and a splitting $N_x=N^u_x\oplus N^s_x$ such that for all $x\in \Lambda$ we have:
$$\|(P_{X}^{t}(x))^{-1}|_{N^{u}_{x}}\|<C\lambda^{t}\text{ and }\|P_{X}^{t}(x)|_{N^{s}_{x}}\|<C\lambda^{t}.$$

We say that $X \in \mathcal{G}_{\mu}^1(M)$ if there exists a neighborhood $\mathcal{V}$ of $X$ in $\mathfrak{X}_\mu^{1}(M)$ such that any $Y\in\mathcal{V}$, has  all the closed orbits and all the singularities of hyperbolic-type.

The following result, which is important \emph{per se}, will be crucial to obtain Theorem~\ref{teo1}.

\begin{maintheorem}\label{teo2}
 If $X\in int_{C^1}(\mathpzc{TS}_\mu(M))$ then $X\in \mathcal{G}^1_{\mu}(M)$.
\end{maintheorem}

The next result was proved recently by Ferreira (\cite{F}) and is a generalization of a  three-dimensional theorem by the authors (\cite{BeRo3}).
\begin{maintheorem}(Ferreira~\cite{F})\label{Ferreira}
If $X\in \mathcal{G}^1_{\mu}(M)$ then $X$ is Anosov.
\end{maintheorem}

Theorem~\ref{teo1} is a direct consequence of Theorem ~\ref{teo2} and Theorem~\ref{Ferreira} and for this reason we just have to concentrate on the proof of Theorem~\ref{teo2}.


\end{section}

\begin{section}{Perturbation lemmata}
\begin{subsection}{Perturbations near singularities}

Some key results to perform perturbations in the conservative setting are available (see~\cite{AM}). Nevertheless, neither \cite[Theorem 3.1]{AM} nor \cite[Theorem 3.2]{AM} are adequate to go on with the proof of our result. Therefore, we need to obtain an upgrade of these pasting lemmas and this is the content of the Lemma~\ref{am}. We would like to thank Carlos  Matheus for a valuable suggestion for the proof of this lemma.  

\begin{lemma}\label{am} Let $M$ be a compact and boundaryless Riemannian manifold of dimension $\geq 2$. Given $\epsilon>0$, $X\in\mathfrak{X}^1_{\mu}(M)$, a compact $\mathcal{K}\subset M$ and an open neighborhood $\mathcal{U}$ of $\mathcal{K}$,  there are $\delta>0$  and an open set $\mathcal{K} \subset \mathcal{V} \subset \mathcal{U}$ such that, if  $Y\in\mathfrak{X}^2_{\mu}(M)$ is $\delta$-$C^1$-close to $X$ in $\mathcal{U}$, then there exists $Z\in \mathfrak{X}^1_{\mu}(M)$  satisfying
\begin{enumerate}
 \item [a)] $Z=Y$ in $\mathcal{V}$,
 \item [b)] $Z$ is $\epsilon$-$C^1$-close to $X$ and
 \item [c)] $Z=X$ outside $\mathcal{U}$.
\end{enumerate}
\end{lemma}

\begin{proof}
We consider $\mathcal{V}\supset \mathcal{K}$ such that $\partial \mathcal{V}$ is $C^\infty$, $\overline{\mathcal{V}}\subset \mathcal{U}$ and $\{\mathcal{U},\text{int}(M\setminus \overline{\mathcal{V}})\}$ is an open covering of $M$. Let $\alpha\colon M\rightarrow[0,1]$ be a $C^\infty$ function such that $\alpha=1$ in $\mathcal{V}$, $\alpha=0$ outside $\mathcal{U}$ and $|\nabla\alpha|\leq K$, where $K$ is a positive constant depending only on $\mathcal{U}$ and $\mathcal{V}$. Define
\begin{equation}
 Z_0(x) := \alpha(x) Y(x) + (1-\alpha(x)) X(x),
\end{equation}
where $Y\in\mathfrak{X}^2_{\mu}(M)$ is $\delta$-$C^1$-close to $X$ on a small open neighborhood $\mathcal{U}$ of $\mathcal{K}$, where $\delta>0$ will be fixed in (\ref{delta}). We observe that $Z_0=Y$ inside $\mathcal{K}$ and $Z_0=X$ outside $\mathcal{U}$, and thus $\nabla\cdot Z_0=0$ in the closed set $\overline{\mathcal{V}}\cup (M\setminus\mathcal{U})$. However, although $\nabla\cdot Z_0$ is close to zero in $\mathcal{U}\setminus\mathcal{V}$, in general $\nabla\cdot Z_0\not=0$. Actually,

\begin{eqnarray*}
\nabla\cdot Z_0&=& (\nabla\alpha)\cdot Y+\alpha(\nabla\cdot Y)  -(\nabla\alpha)\cdot  X+(1-\alpha)\nabla\cdot Y\\
&=& (\nabla\alpha)\cdot Y -(\nabla\alpha)\cdot X= (\nabla\alpha)\cdot (Y- X),
\end{eqnarray*}
and we have $|\nabla\cdot Z_0|<K\delta$.

Now we will make use of \cite[Theorem 2]{DM} in order to obtain $Z_1\in\mathfrak{X}^2_{\mu}(M)$ (supported in $\mathcal{U}\setminus\mathcal{V}$) such that $\nabla\cdot Z_1=-\nabla\cdot Z_0$ and $Z_1=0$ in $\partial(\mathcal{U}\setminus\mathcal{V})$.

Finally, we define
\begin{equation}
Z:=Z_0+Z_1. 
\end{equation}
Of course that, by construction, $\nabla\cdot Z=0$ and c) holds. 

In $\mathcal{U}\setminus\mathcal{V}$ we have
\begin{eqnarray*}
 \|Z-Y\|_{C^1}&=&\|Z_0+Z_1-Y\|_{C^1}\leq \|Z_1\|_{C^1}+\|Z_0-Y\|_{C^1}\\
&\leq& C\|\nabla\cdot Z_1\|_{C^0}+\|Z_0-Y\|_{C^1},
\end{eqnarray*}
where $C>0$ is a constant given in Dacorogna-Moser theorem (see~\cite[Theorem 2.3]{AM}) that depends only on $\mathcal{U}$. 

Going back to the beginning of the proof, we take
\begin{equation}\label{delta}
\delta<\min\left\{\epsilon,\frac{\epsilon}{2CK}\right\}.
\end{equation}

Now, using $|\nabla\cdot Z_0|<K\delta$ we get,
\begin{eqnarray*}
\|Z-Y\|_{C^1}&\leq& C\|\nabla\cdot Z_1\|_{C^0}+\|Z_0-Y\|_{C^1} \leq CK\delta+\|\alpha Y + (1-\alpha) X-Y\|_{C^1}\\
&\leq& \frac{\epsilon}{2}+\| (1-\alpha) X-(1-\alpha) Y\|_{C^1}\leq \frac{\epsilon}{2}+|1-\alpha| \|X- Y\|_{C^1}\\
&\leq& \frac{\epsilon}{2}+\|X- Y\|_{C^1}\leq \frac{\epsilon}{2}+\delta<\epsilon.
\end{eqnarray*}

\end{proof}

The next lemma is the divergence-free vector fields version of \cite[Lemma 1.1]{MSS}.

\begin{lemma}\label{linear}
Let $\sigma$ be a singularity of $X \in \mathfrak{X}_\mu^1(M)$. For any $\epsilon >0$ there exist $\delta_0>0$ and $\epsilon_0>0$ such that if $Y_\delta \colon T_\sigma M \rightarrow T_\sigma M$ is a traceless linear map $\frac{\delta}{2}$-$C^0$-close to $DX_\sigma$ (with $\delta<\delta_0$) then there exists $Z_\delta \in  \mathfrak{X}_\mu^{1}(M)$, such that $Z_\delta=Y_\delta$ in $B_{\epsilon_0/4}(\sigma)$, $Z_\delta$ is $\epsilon$-$C^1$-close to $X$ and $Z_\delta=X$ outside the set $B_{\epsilon_0}(\sigma)$.
\end{lemma}
\begin{proof}

Let $(U,\phi)$ be a conservative chart given by Moser's Theorem (\cite{Mo}) such that $\sigma \in U$ and $\phi(\sigma)=\textbf{0}$. Now we will work on the euclidean space $\mathbb{R}^n$.

We fix $\delta\in(0,\delta_0)$ where $\delta_0>0$ will be defined in the sequel. For simplicity we assume that $Y_\delta \colon \mathbb{R}^n \rightarrow \mathbb{R}^n$ is written in the canonical coordinates. Now, we consider the divergence-free linear vector field in $\mathbb{R}^n$ defined (in the canonical coordinates) by;
\begin{equation}\label{Y}
 (\dot{x}_1,...,\dot{x}_n)=Y_\delta (x_1,...,x_n).
\end{equation}

Now, let $\mathcal{K}=\overline{B_{\epsilon_0/4}(\textbf{0})}$ and $\mathcal{U}=B_{\epsilon_0}(\textbf{0})$. Since, by hypothesis, the map $Y_\delta$ is $\frac{\delta}{2}$-$C^0$-close to $DX_\textbf{0}$, if we choose $\epsilon_0$ very small, then $X$ is $\frac{\delta_0}{2}$-$C^1$-close to $DX_{\textbf{0}}$ when restricted to $\mathcal{U}$.  Therefore $Y_\delta$ (defined in (\ref{Y})) and $X$ are $\delta_0$-$C^1$-close in $\mathcal{U}$. This conditions gives $\epsilon_0$ depending on $\delta_0$. So, by Lemma~\ref{am}, fixed $\epsilon$, there exist $\delta_0$ and there are an open set $\mathcal{V}$; $\mathcal{K} \subset \mathcal{V} \subset \mathcal{U}$ and $Z_\delta\in \mathfrak{X}^1_{\mu}(M)$  such that $Z_\delta=Y_\delta$ in $\mathcal{V}$, $Z_\delta$ is $\epsilon$-$C^1$-close to $X$ and $Z_\delta=X$ outside $\mathcal{U}$ for $\delta<\delta_0$. The lemma is proved.
\end{proof}

\end{subsection}

\begin{subsection}{Perturbations near closed orbits}
 
Lemma~\ref{mpl} below is a \emph{Franks' lemma} for incompressible flows.

Define $\Gamma(p, \tau)=\{X^t(p);\,\, t \in [0, \tau]\}$. Let $V, \tilde{V} \subset N_p$, $dim(V)=j$, $2 \leq j \leq n-1$, and  $N_p=V \oplus \tilde{V}$. A \emph{one-parameter linear family} $\{A_t\}_{t\in \mathbb{R}}$ associated to $\Gamma(p, \tau)$ and $V$  is defined as follows:
\begin{itemize}
\item $A_t\colon N_p \rightarrow N_p$ is a linear map, for all $t\in \mathbb{R}$,
\item $A_t=id$, for all $t\leq 0$, and $A_t=A_{\tau}$, for all $t\geq \tau$,
\item $A_t|_V \in sl(j, \mathbb{R})$, and $A_t|_{\tilde{V}}\equiv id$, $\forall t \in [0, \tau]$, in particular we have  $\det(A_t)=1$, for all $t\in \mathbb{R}$, and
\item the family $A_t$ is $C^\infty$ on the parameter $t$.
\end{itemize}

\begin{lemma}\label{mpl}(\cite[Lemma 3.2]{BeRo2}) Given $\epsilon>0$ and a vector field $X \in \mathfrak{X}_\mu^4(M)$ there exists $\theta_0=\theta_0(\epsilon,X)$ such that $\forall \tau \in [1,2]$, for any periodic point $p$ of period greater than $2$, for any sufficient small flowbox  $\mathcal{T}$ of $\Gamma(p, \tau)$ and for any one-parameter linear family $\{A_t\}_{t \in [0, \tau]}$ such that $\|  \dot{A}_t A_t^{-1}\|<\theta_0$, $\forall t \in [0, \tau]$, there exists $Y \in \mathfrak{X}_\mu^1(M)$ satisfying the following properties
\begin{enumerate}
\item [(A)] $Y$ is $\epsilon$-$C^1$-close to $X$;
\item [(B)] $Y^t(p)=X^t(p)$, for all $t \in \mathbb{R}$;
\item [(C)] $P_Y^\tau(p)=P_X^\tau(p) \circ A_{\tau}$, and
\item [(D)] $Y|_{\mathcal{T}^c}\equiv X|_{\mathcal{T}^c}$.
\end{enumerate}
\end{lemma}
\end{subsection}
\end{section}

\begin{section}{Proof of Theorem~\ref{teo2}}

In order to go on with our proof we observe that the main steps are based on the arguments in \cite[\S 2]{MSS}. From \S\ref{sing} and \S\ref{closed} below it follows that any $X\in int_{C^1}(\mathpzc{TS}_\mu(M))$ has all its singularities and closed orbits of hyperbolic-type. Therefore, $int_{C^1}(\mathpzc{TS}_\mu(M))\subset \mathcal{G}^1_\mu(M)$.

\begin{subsection}{Singularities}\label{sing}
We are going to prove that, if $X\in int_{C^1}(\mathpzc{TS}_\mu(M))$, then any singularity of $X$ is hyperbolic. By contradiction let us assume that there exists a non-hyperbolic singularity $\sigma \in Sing(X)$. According to Lemma~\ref{linear} we consider a family of divergence-free vector fields $\{Z_\delta\}_{\delta\geq 0}$, where $Z_{\delta}$ is linear and $Z_0=DX_{\sigma}$ in $B_{\frac{\epsilon_0}{4}}(\sigma)$, and $\delta$ is sufficiently small to assure that this family is contained in $int_{C^1}(\mathpzc{TS}_\mu(M))$. We observe that 
we can chose $\delta^\prime$ arbitrarily small such that $\sigma$ is a hyperbolic singularity for $Z_{\delta^\prime}$ and such that $Z_{\delta^{\prime}}$ is semiconjugated to $Z_0$, that is, there exists a continuous and onto map $h\colon M\rightarrow M$ (arbitrarily close to the identity depending on $\delta^{\prime}$) and a continuous real map $\tau \colon M\times\mathbb{R}\rightarrow\mathbb{R}$ such that for any $x\in M$, $\tau_x \colon \mathbb{R}\rightarrow\mathbb{R}$ is an orientation preserving homeomorphism where $\tau(x,0)=0$ and for all $x\in M$ and $t\in\mathbb{R}$ we have
\begin{equation}\label{h}
 h(Z_{\delta^\prime}^t(x))=Z_{0}^{\tau(x,t)}(h(x)).
\end{equation}

As $\sigma$ is non-hyperbolic for $Z_0$, there exists $z\in M$ such that $$\sigma\notin \{B_{\epsilon}(Z_{0}^{t}(z))\colon t\in\mathbb{R}\} \text{ and } \{B_{\epsilon}(Z_{0}^{t}(z))\colon t\in\mathbb{R}\}\subset B_{\frac{\epsilon_0}{8}}(\sigma)$$ and $\epsilon<\frac{\epsilon_0}{16}$. Let $w\in h^{-1}(z)$. From (\ref{h}) we get that $h(Z_{\delta^\prime}^t(w))=Z_0^{\tau(w,t)}(z)$ and, since $h$ is arbitrarily close to the identity, we obtain that $$\{Z_{\delta^\prime}^{t}(w)\colon t\in\mathbb{R}\}\subset \{B_{\epsilon}(Z_{0}^{t}(z))\colon t\in\mathbb{R}\}\subset B_{\frac{\epsilon_0}{8}}(\sigma),$$
which is a contradiction because $\sigma$ is a hyperbolic singularity of $Z_{\delta^\prime}$ and, when restricted to $B_{\frac{\epsilon_0}{8}}(\sigma)$, the vector field $Z_{0}$ is linear.
\end{subsection}

\begin{subsection}{Closed orbits}\label{closed}
Fix $X\in int_{C^1}(\mathpzc{TS}_\mu(M))$. Now we are going to prove that all the closed orbits of $X$ are hyperbolic. 

Assume that $X$ has a non-hyperbolic closed orbit $p$ of period $\pi(p)$. In order to proceed with the arguments in \cite[\S 2]{MSS} we need to $C^1$-approximate the vector field $X$ by a new one which is \emph{linear} in a neighborhood of the closed orbit $p$. To perform this task, in the conservative setting, it is required more differentiability of the vector field (cf. Lemma~\ref{mpl}).

If $X$ is of class $C^\infty$, take $Z=X$, otherwise we use \cite{Z} in order to obtain a $C^\infty$ incompressible vector field $Y\in int_{C^1}(\mathpzc{TS}_\mu(M))$, arbitrarily $C^1$-close to $X$, and such that $Y$ has a closed orbit\footnote{Notice that $q$ may not be the analytic continuation of $p$. This is precisely the case when $1$ is an eigenvalue of $P_X^{\pi(p)}(p)$.} $q$, close to $p$, and with period $\pi(q)$ close to $\pi(p)$. If $q$ is not hyperbolic take $Z=Y$. If $q$ is hyperbolic for $P_Y^{\pi(q)}(q)$, then, since $Y$ is $C^1$-arbitrarily close to $X$, the distance between the spectrum of $P_Y^{\pi(q)}(q)$ and $\mathbb{S}^1$ can be taken arbitrarily close to zero (weak hyperbolicity). Now, we are in position to apply Lemma~\ref{mpl} to obtain a new vector field $Z\in\mathfrak{X}^\infty_\mu(M)\cap int_{C^1}(\mathpzc{TS}_\mu(M))$, $C^1$-close to $Y$ and having a non-hyperbolic closed orbit.

Now, we argue as in \cite[\S 3]{BeRo1} in order to obtain $L\in int_{C^1}(\mathpzc{TS}_\mu(M))$ such that $L$ is linear (equal to $P_Z^t$) in a neighborhood of the closed non-hyperbolic orbit, $\Gamma$.

Finally, we $C^1$-approximate $L$ by $W\in int_{C^1}(\mathpzc{TS}_\mu(M))$ such that $\Gamma$ is hyperbolic (for $W$). This is a contradiction because $W$ is semiconjugated to $L$, although there is an $L^t$-orbit (different from $\Gamma$) contained in a small neighborhood of $\Gamma$ and the same cannot occur for $W^t$ because $\Gamma$ is a hyperbolic closed orbit for $W^t$.

\end{subsection}

\end{section}

\begin{section}{The volume-preserving diffeomorphisms case}\label{diff}
Let $\text{Diff}_\mu^{\,\,1}(M)$ denote the set of volume-preserving (or conservative) diffeomorphisms defined on $M$, and consider this space endowed with the $C^1$ Whitney topology. In this section we assume that $\dim(M)\geq 2$. We say that a diffeomorphism $f$ is \emph{Anosov} if, there exist $\lambda\in(0,1)$ and $C>0$ such that the tangent vector bundle over $M$ splits into two $Df$-invariant subbundles $TM=E^u\oplus E^s$, with $\|Df^n|_{E^s}\|\leq C\lambda^n$ and $\|Df^{-n}|_{E^u}\|\leq C\lambda^n$.

We say that $f \in \mathcal{F}_{\mu}^1(M)$ if there exists a neighborhood $\mathcal{V}$ of $f$ in $\text{Diff}_\mu^{\,\,1}(M)$ such that any $g\in\mathcal{V}$ has all the periodic orbits  hyperbolic. In~\cite{AC} Arbieto and Catalan proved the following result.

\begin{maintheorem}\label{AC}
If $f\in \mathcal{F}^1_{\mu}(M)$ then $f$ is Anosov. 
\end{maintheorem}

Given $f,g\in\text{Diff}_\mu^{\,\,1}(M)$ we say that $g$ is \emph{semiconjugated} to $f$ if there exists a continuous and onto map $h\colon M\rightarrow M$ such that for all $x\in M$ one has $h(g(x))=f(h(x))$. 

We say that $f$ is \emph{topologically stable} in $\text{Diff}_\mu^{\,\,1}(M)$ if, for any $\epsilon>0$, there exists $\delta>0$ such that for any $g\in \text{Diff}_\mu^{\,\,1}(M)$ $\delta$-$C^0$-close to $f$, there exists a semiconjugacy from $g$ to $f$, i.e., there exists $h\colon M\rightarrow M$ satisfying $h(g(x))=f(h(x))$ and $d(h(x),x)<\epsilon$, for all $x\in M$. Once again we emphasize that our definition of topological stability is restricted to the conservative setting and not to the  broader space of dissipative diffeomorphisms. We denote the set of topological stable incompressible flows by $\emph{\textbf{TS}}_\mu(M)$.

In this section we obtain the discrete-time version of Theorem~\ref{teo1}.

\begin{maintheorem}
If $f\in int_{C^1}(\textbf{TS}_\mu(M))$ then $f$ is Anosov. 
\end{maintheorem}

The proof is similar to the one done in \S\ref{closed} and, for this reason, we present a brief highlight of it. As before, given $f\in int_{C^1}(\textbf{TS}_\mu(M))$, we prove that all its  periodic orbits are hyperbolic; from this it follows that $int_{C^1}(\textbf{TS}_\mu(M))\subset \mathcal{F}^1_{\mu}(M)$. Then, using Theorem~\ref{AC}, we obtain that any $f\in int_{C^1}(\textbf{TS}_\mu(M))$ is Anosov. 

Once again we assume, by contradiction, that there is some non-hyperbolic orbit. Now, to argue as in the flow case, we make use of the following two ingredients:
\begin{enumerate}
 \item a way to linearize the diffeomorphism in a neighborhood of a periodic point and
 \item a ``perturbation of the derivative'' result in the vein of Lemma~\ref{mpl}.
\end{enumerate}

The item (2) is available in the literature (see~\cite[Proposition 7.4]{BDP}). 

For (1) we just have to approximate $f\in int_{C^1}(\emph{\textbf{TS}}_\mu(M))$ by a diffeomorphism $g\in int_{C^1}(\emph{\textbf{TS}}_\mu(M))\cap \text{Diff}_\mu^{\,\,\infty}(M)$ using a recent result by Avila~(\cite{A}), and then we use the Pasting lemma~\cite[Theorem 3.6]{AM} to obtain $h\in int_{C^1}(\emph{\textbf{TS}}_\mu(M))\cap \text{Diff}_\mu^{\,\,\infty}(M)$ such that $h=Dg_p$ in a neighborhood of the periodic orbit $p$. This is precisely what we need to obtain a contradiction as we did in \S\ref{closed}.

\end{section}

\section*{Acknowledgements}

The authors were partially supported by the FCT-Funda\c{c}\~ao para a Ci\^encia e a Tecnologia, project PTDC/MAT/099493/2008.
MB was partially supported by FCT (SFRH/BPD/20890/2004).



\begin{thebibliography}{ABC}





\bibitem{AC} A. Arbieto and T. Catalan, \emph{Hyperbolicity in the Volume Preserving Scenario},  arXiv:1004.1664, Preprint 2010.

\bibitem{AM} A. Arbieto and C. Matheus, \emph{A pasting lemma and some applications for conservative systems}. With an appendix by David Diica and Yakov Simpson-Weller., Ergod. Th. \&\ Dynam. Sys., 27 (2007), 1399--1417.


\bibitem{A} A. Avila, \emph{On the regularization of conservative maps}, arXiv:0810.1533 Preprint, to appear in Acta Math.


\bibitem{BeRo2} M. Bessa and J. Rocha, \emph{On $C^1$-robust transitivity of volume-preserving flows}, Jr. Diff. Eq., 245, 11, (2008), 3127--3143.

\bibitem{BeRo1} M. Bessa and J. Rocha, \emph{Homoclinic tangencies versus uniform hyperbolicity for conservative 3-flows}, Jr. Diff. Eq., 247, 11, (2009), 2913--2923.

\bibitem{BeRo3} M. Bessa and J. Rocha, \emph{Three-dimensional  conservative star flows are Anosov}, Disc. Cont. Dyn. Sys. A, 26, 3, (2010), 839--846.


\bibitem{BDP} C. Bonatti, L.J. D\'{i}az and E. Pujals, \emph{A $C^1$-generic dichotomy for diffeomorphisms: Weak forms of hyperbolicity or infinitely many sinks or sources}, Ann. Math, 158 (2005), 355--418.



\bibitem{DM} B. Dacorogna and J. Moser, \emph{On a partial differential equation involving the Jacobian determinant}, Ann. Inst. Henri Poincar\'{e}, 7, 1, (1990) 1--26.

\bibitem{F} C. Ferreira, \emph{Stability properties of divergence-free vector fields}, arXiv:1004.2893, Preprint 2010.



\bibitem{H} M. Hurley, \emph{Fixed points of topologically stable flows},  Trans. Amer. Math. Soc.  294  (1986),  2, 625--633.


\bibitem{H2} M. Hurley, \emph{Consequences of topological stability},  J. Differential Equations  54  (1984),  1, 60--72.

\bibitem{H3} M. Hurley, \emph{Combined structural and topological stability are equivalent to Axiom A and the strong transversality condition},  Ergod. Th. \&\ Dynam. Sys., 4, 1 (1984), 81--88. 

\bibitem{Mo} J. Moser, {On the volume elements on a manifold}. Trans. Amer. Math. Soc., 120 (1965), 286--294.

\bibitem{PT} J. F. Plante and W. Thurston, \emph{Anosov flows and the fundamental group},  Topology  11  (1972), 147--150.



\bibitem{Ro2} C. Robinson, \emph{Stability theorems and hyperbolicity in dynamical systems}, Rocky Mountain J. Math., 7, 3, (1977), 425--437.

\bibitem{MSS} K. Moriyasu, K. Sakai and N. Sumi, \emph{Vector fields with topological stability}, Trans. Amer. Math. Soc., 353, 8 (2001), 3391--3408.

\bibitem{Ni} Z. Nitecki, \emph{On semi-stability for diffeomorphisms}, Invent. Math. 14 (1971), 83--122.


\bibitem{R} R. O. Ruggiero, \emph{Topological stability and Gromov hyperbolicity}, Ergod. Th. \&\ Dynam. Sys., 19 (1999), 143--154.


\bibitem{W} P. Walters, \emph{Anosov diffeomorphisms are topologically stable}, Topology 9, (1970), 71--78.


\bibitem{Z} C. Zuppa, \emph{Regularisation $C^{\infty}$ des champs vectoriels qui pr\'{e}servent l'el\'{e}ment de volume}, Bol. Soc. Bras. Mat., 10, 2 (1979), 51--56.
\end{thebibliography}
\end{document}